\documentclass[12pt]{article}

\usepackage{amsmath}
\usepackage{amsfonts,amssymb,amsthm,overpic}
\makeatletter
\newcommand{\diam}{\mathop{\operator@font diam}}
\usepackage{setspace}
\usepackage{multicol}
\usepackage{multirow}
\usepackage[compact]{titlesec}

\begin{document}

\def\s{\subseteq}
\def\h{\widehat}
\def\v{\varphi}
\def\t{\widetilde}
\def\ov{\overline}
\def\L{\Lambda}
\def\l{\lambda}
\def\O{\Omega}
\def\H{I\!\! H}
\def\a{\approx}
\def\k{\widetilde}
\def\la{\lambda}
\def\d{\delta}
\def\L{\Lambda}
\def\O{\Omega}
\def\r{\rho}
\def\ov{\overline}
\def\un{\underline}
\newcommand{\cT}{\mathcal{T}}
\newcommand{\cR}{\mathcal{R}}
\newcommand{\cL}{\mathcal{L}}
\newcommand{\cU}{\mathcal{U}}
\newcommand{\cV}{\mathcal{V}}
\newcommand{\cW}{\mathcal{W}}
\newcommand{\cM}{\mathcal{M}}
\newcommand{\cN}{\mathcal{N}}
\newcommand{\cS}{\mathcal{S}}
\newcommand{\cP}{\mathcal{P}}
\newcommand{\cH}{\mathcal{H}}
\newcommand{\FR}{{{}^\bullet}\mathbb{R}}
\newcommand{\triangleL}{\triangleleft_{\mathcal{L}}}
\newcommand{\triangleR}{\triangleright_{\mathcal{R}}}
\newcommand{\trangleR}{\triangleleft_{\mathcal{R}}}
\newcommand{\trangleL}{\triangleright_{\mathcal{L}}}
\newcommand{\triangleH}{\triangleleft_{\mathcal{H}}}

\newcommand{\triangleLq}{\trianglelefteq_{\mathcal{L}}}
\newcommand{\triangleRq}{\trianglelefteq_{\mathcal{R}}}
\newcommand{\trangleRq}{\trianglerighteq_{\mathcal{R}}}

\title{{\bf A Few Questions on the Topologization of the Ring of Fermat Reals.}}

\author{Kyriakos Papadopoulos, kyriakos.papadopoulos1981@gmail.com\\
\small{American University of the Middle East, Kuwait,}}

\date{}
\maketitle

\begin{abstract}
This work is developing, and we will include many additions in the
near future. Our purpose here is to highlight that there is plenty
of space for a topological development of the Fermat Real
Line.
\end{abstract}



\newtheorem{definition}{Definition}[section]
\newtheorem{remark}{Remark}[section]
\newtheorem{remarks}{Remarks}[section]
\newtheorem{notation}{Notation}[section]
\newtheorem{examples}{Examples}[section]
\newtheorem{example}{Example}[section]
\newtheorem{theorem}{Theorem}[section]
\newtheorem{proposition}{Proposition}[section]
\newtheorem{co}{Corollary}[section]
\newtheorem{lemma}{Lemma}[section]
\newtheorem{corollary}{Corollary}[section]

\section{Introduction}

The idea of the ring of Fermat Reals $\FR$ has come as a possible alternative to Synthetic
Differential Geometry (see e.g. [11, 12, 13]) and its main aim is the development
of a new foundation of smooth differential geometry for finite and infinite-dimensional
spaces. In addition, $\FR$ could play a role of a potential alternative in some certain
 problems in the field ${}^\star \mathbb{R}$ in Nonstandard Analysis (NSA), because
 the applications of NSA in differential geometry are very few.
One of the ``weak'' points of $\FR$ at the moment is
the lack of a natural topology, carrying the strong topological
properties of the line.

P. Giordano and M. Kunzinger have recently done brave steps
towards the topologization of the ring ${}^\bullet \mathbb{R}$ of
Fermat Reals. In particular, they have constructed two topologies;
the Fermat topology and the omega topology (see \cite{ref3}). The Fermat topology is generated
by a complete pseudo-metric and is linked to the differentiation of non-standard
smooth functions. The omega topology is generated by a complete metric and
is linked to the differentiation of smooth functions on infinitesimals.
Although both topologies are very useful in developing infinitesimal
instruments for smooth differential geometry, none of these
two topologies aims to characterize the Fermat real line from an
order-theoretic perspective. In fact, neither makes the space $T_1$, while
an appropriate order-topology would equip the Fermat Real Line
with the structure of a monotonically normal space, at least. The
possibility to define a linear order relation on $\FR$, so that it
can be viewed as a LOTS (linearly ordered topological space) can
be considered important, because $\FR$ is an alternative mathematical
model of the real line, having some features with respect to
applications in smooth differential geometry and mathematical physics.
It is therefore natural to ask whether for $\FR$ peculiar characteristics
of $\mathbb{R}$ hold or not.

In this paper we will focus in the order relation which is introduced
in \cite{ref4} (which is linear, but it generates the discrete topology
on the space and also if considered minus the diagonal, i.e. as a strict
order, the topology when restricted to the set of proper reals is again
the discrete topology) and we will add properties to it, so that it will both extend
the natural order of the real line and it will also give a stronger topology
than the Fermat topology and the omega topology. We aim to do this using
interlocking nests.

As we shall see in Definition \ref{definition-equivalens relation Fermat Real},
the idea of the formation of $\FR$ starts with an equivalence relation in
the little-oh polynomials, where $\FR$ is the quotient space under this
relation. This treatment permits us to view these little-oh polynomials as
numbers.

\section{Preliminaries.}
\subsection{LOTS and GO-spaces via Nests.}

The notions nest and order are closely related. J. van Dalen and E. Wattel gave
a complete characterization of LOTS (linearly ordered topological spaces) and of
GO-spaces (generalized ordered spaces, i.e. subspaces of LOTS) using properties of nests (see \cite{ref1}).
In this paper we will use tools from \cite{ref2}, where the authors improved the techniques
of van Dalen and Wattel in order to characterize ordinals in topological terms and from \cite{ref5}, where the author studies further
properties of order relations via nests.

\begin{definition}\label{definition-T0-T1}
Let $X$ be a set.
\begin{enumerate}
\item A collection $\mathcal{L}\subset \mathcal{P}(X)$ of subsets of $X$ $T_0$-{\em separates} $X$, if and only if for every $x,y \in X$, such that $x \neq y$, there exists $L \in
\mathcal{L}$, such that $x \in L$ and $y \notin L$ or $y \in L$
and $x \notin L$.

\item A collection $\mathcal{L} \subset \mathcal{P}(X)$ of subsets of $X$ $T_1$-{\em separates} $X$, if and only if for every $x,y \in X$, such that $x \neq y$, there exist $L,L' \in
\mathcal{L}$, such that $x \in L$ and $y \notin L$ and also $y \in
L'$ and $x \notin L'$.
\end{enumerate}
\end{definition}

\begin{definition}
Let $X$ be a set and let $\mathcal{L}$ be a family
of subsets of $X$. $\mathcal{L}$ is a {\em nest}
on $X$ if, for every $M, N \in \mathcal{L}$, either $M \subset N$ or $N \subset M$.
\end{definition}

\begin{definition}\label{definition-order-reflexive}
Let $X$ be a set and $x, y \in X$. We declare $x \triangleL y$,
if and only if there exists $L \in \cL$, such that $x \in L$
and $y \notin L$.
\end{definition}

It follows that $x \triangleLq y$, if and only if either $x \triangleL y$ or $x = y$. One can
easily see that $\triangleLq$ is a linear order, if $\cL$ is a $T_0$-separating nest.

\begin{theorem}[(See \cite{ref2})]\label{theorem - preliminary to Lemma for Theorem Dalen-Wattel}
Let $X$ be a set and suppose that $\cL$ and $\cR$ are two nests on $X$. Then,
$\cL \cup \cR$ is $T_1$-separating, if and only if $\cL$ and $\cR$ are
both $T_0$-separating and $\triangleL = \triangleR$.
\end{theorem}

\begin{definition}[van Dalen \& Wattel]\label{definiton-interlocking}
Let $X$ be a set and let $\cL \subset \cP(X)$. We say that $\cL$ is interlocking
if and only if, for each $L \in \cL$, $L = \bigcap\{N \in \cL : L \subsetneq N\}$
implies $L = \bigcup \{N \in \cL : N \subsetneq L\}$.
\end{definition}

\begin{theorem}[See \cite{ref2}]\label{theorem-interlocking}
Let $X$ be a set and let $\cL$ be a $T_0$-separating nest on $X$. The following are equivalent:
\begin{enumerate}

\item $\cL$ is interlocking;

\item for each $L \in \cL$, if $L$ has a $\triangleL$-maximal element, then $X-L$ has a $\triangleL$-minimal element;


\end{enumerate}
\end{theorem}

\begin{theorem}[van Dalen \& Wattel]\label{theorem - them Dalen-Wattel}
Let $(X,\mathcal{T})$ be a topological space. Then:
\begin{enumerate}

\item If $\mathcal{L}$ and $\mathcal{R}$ are two nests of open
sets, whose union is $T_1$-separating, then every $\triangleleft_{\mathcal{L}}$-order
open set is open, in $X$.

\item $X$ is a GO-space, if and only if there are two nests $\mathcal{L}$ and $\mathcal{R}$
of open sets, whose union is $T_1$-separating and forms a subbasis for $\cT$.

\item $X$ is a LOTS, if and only if there are two interlocking nests $\cL$ and $\cR$
of open sets, whose union is $T_1$-separating and forms a subbasis for $\cT$.

\end{enumerate}
\end{theorem}

\subsection{The Ring ${}^\bullet \mathbb{R}$ of Fermat Reals.}

The material in this subsection can be found in \cite{ref7}, \cite{ref6} and also in \cite{ref4}.

\begin{definition}
A little-oh polynomial $x_t$ (or $x(t)$) is an ordinary set-theoretical function, defined as follows:
\begin{enumerate}

\item $x : \mathbb{R}_{\ge 0} \to \mathbb{R}$ and

\item $x_t = r + \sum_{i=1}^k \alpha_i t^{a_i} + o(t)$, as $t \to 0^+$, for
suitable $k \in \mathbb{N}$, $r,\alpha_1,\cdots,\alpha_k \in \mathbb{R}$ and $a_1,\cdots,a_k \in \mathbb{R}_{\ge 0}$.
\end{enumerate}
\end{definition}
The set of all little-oh polynomials is denoted by the symbol $\mathbb{R}_o[t]$. So, $x \in \mathbb{R}_o(t)$,
if and only if $x$ is a polynomial function with real coefficients, of a real variable $t \ge 0$, with
generic positive powers of $t$ and up to a little-oh function $o(t)$, as $t \to 0^+$.

\begin{definition}\label{definition-equivalens relation Fermat Real}
Let $x,y \in \mathbb{R}_o[t]$. We declare $x\sim y$ (and we say $x=y$ in ${}^\bullet \mathbb{R}$), if and only
if $x(t) = y(t) + o(t)$, as $t \to 0^+$.
\end{definition}

The relation $\sim$ in Definition \ref{definition-equivalens relation Fermat Real} is an equivalence
relation and ${}^\bullet \mathbb{R} :=\mathbb{R}_o[t]/\sim$.

A first attempt to define an order in ${}^\bullet \mathbb{R}$ has come from Giordano.

\begin{definition}[Giordano]\label{definition-order Giordano}
Let $x,y \in {}^\bullet \mathbb{R}$. We declare $x \le y$, if and only if there exists $z \in {}^\bullet \mathbb{R}$,
such that $z = 0$ in ${}^\bullet \mathbb{R}$ (i.e. $\lim_{t \to 0^+} z_t/t = 0$) and for
every $t \ge 0$ sufficiently small, $x_t \le y_t + z_t$.
\end{definition}

For simplicity, one does not use equivalence relation but works with an equivalent language of
representatives. If one chooses to use the notations of \cite{ref4}, one has to note that
Definition \ref{definition-order Giordano} does not depend on representatives.

As the author describes in \cite{ref4}, the order relation in NSA admits all formal properties
among all the theories of (actual) infinitesimals, but there is no good dialectic of these
properties with their informal interpretation. In particular, the order in ${}^\star \mathbb{R}$
inherits by transfer all the first order properties but, on the other hand, in the quotient
field ${}^\star \mathbb{R}$ it is difficult to interpret these properties of the order
relation as intuitive properties of the corresponding representatives. For example, a geometrical
interpretation like that of $\FR$ seems not possible for ${}^ \star \mathbb{R}$. Definition \ref{definition-order Giordano}
provides a clear geometrical representation of the ring $\FR$ (see, for instance, section 4.4 of \cite{ref4}).

\subsection{The Fermat Topology and the omega-topology on $\FR$.}

A subset $A \subset \FR^n$ is open in the Fermat topology, if it can be written as
$A = \bigcup \{{}^\bullet U \subset A: U \textrm{ is open in the natural topology in }\mathbb{R}^n\}$.
Giordano and Kunzinger describe this topology as the best possible one for sets having a ``sufficient
amount of standard points'', for example ${}^\bullet U$. They add that this connection between the
Fermat topology and standard reals can be glimpsed by saying that the monad $\mu(r) := \{x \in \FR : \textrm{ standard part of }x = r\}$ of a real $r \in \mathbb{R}$
is the set of all points which are limits of sequences with respect to the Fermat topology. However it is obvious that in sets of infinitesimals there is a need for constructing a (pseudo-)metric generating a finer topology that the authors call the omega-topology (see \cite{ref3}).
Since neither the Fermat nor the omega-topology are Hausdorff when restricted to $\FR$ and since each
of them describes sets having a ``sufficient amount'' of standard points or infinitesimals, respectively,
there is a need for defining a natural topology on $\FR$ describing sufficiently all Fermat reals
and carrying the best possible properties.

\section{Interlocking Nests on ${}^\bullet \mathbb{R}$.}

A first disadvantage of the construction in Definition \ref{definition-order Giordano}
is that the order $\le$ in $\FR$ does not generate interlocking nests, missing
points from the Fermat real line. In particular, the nest $\cL$ consisting of sets $L=\{k \in \FR : k \le l\}$, for some $l \in \FR$,
has as maximal element the fermat real $l$, but the complement of $L$, i.e. $L^c = \{k \in \FR : k> l\}$, for some $l \in \FR$,
does not have a minimal element. Thus, we first remark that the order of Definition \ref{definition-order Giordano} makes ${}^\bullet \mathbb{R}$ a GO-space, a subspace of a particular LOTS. So we will now have to construct an appropriate order in ${}^\bullet \mathbb{R}$ which makes it LOTS, by completing the missing minimal elements from complements of sets with maximal elements. Even the fact that $\le$ is
linear, it generates the discrete topology on $\FR$ and, if considered as a strict order, the restriction of
its topology in $\mathbb{R}$ will be again the discrete topology.

\subsection{Order Relations and an Order Topologies on $\FR$.}

\begin{theorem}
The pair $({}^\bullet \mathbb{R},<_F)$, where $<_F$ is defined as follows:

\[
x<_Fy \Leftrightarrow \begin{cases}
\exists\,\{k \in \FR : k \le l\}, \textrm{ some } l \in \FR, \textrm{ such that } x \in \{k \in \FR : k \le l\} \not\ni y, l \in \FR
 \\
or\\
x = \max\{k \in \FR : k \le l\}, \textrm{ some } l \in \FR \textrm{ and } \exists h \in \FR : h>0,\,y = x + h\\
or\\
y = \min \{k \in \FR : l \le k\}, \textrm{ some } l \in \FR \textrm{ and } \exists h \in \FR : h>0,\,x = y - h
\end{cases}
\]
where $x,y$ are distinct Fermat reals, is a linearly ordered set.
\end{theorem}

\begin{proof}
The order of Definition \ref{definition-order Giordano} gives two nests, namely
the nest $\cL$, which consists of all sets $L = \{k \in \FR: k \le l\}$, some $l \in \FR$ and the nest $\cR$,
which consists of all sets $R = \{k \in \FR : l \le k\}$, some $l \in \FR$.
In addition, Theorem \ref{theorem - preliminary to Lemma for Theorem Dalen-Wattel} implies that
$\triangleLq = \trianglerighteq_{\cR} = \le$.

We remark that, for any $L \in \cL$ (respectively for any $R \in \cR$), $L$ (resp. $R$) has a $\triangleLq$-maximal element (resp. $\triangleRq$-maximal element for $R$), such that $X-L$ has no $\triangleLq$-minimal element (resp. $X-R$ has no $\triangleRq$-minimal element). So,
according to Theorem \ref{theorem-interlocking}, neither $\cL$ nor $\cR$ are interlocking.

Now, for all $L=\{k \in \FR : k \le l\} \in \cL$, some $l \in \FR$, let $x_L$ denote the $\triangleLq$-maximal element of $L$ and for all $R=\{k \in \FR : l \le k\} \in \cR$, some $l \in \FR$
let $y_R$ denote the $\triangleLq$-minimal element of $R$.

Furthermore, for each $L \in \cL$ choose $x_L^+ \in \FR$  and
for each $R \in \cR$ choose $y_R^- \in \FR$, where $x_L^+$ and $y_R^-$ are distinct points in ${}^\bullet \mathbb{R}$,
and define a map $p : {}^\bullet \mathbb{R} \to {}^\bullet \mathbb{R}-(\{x_L^+ : L \in \cL\} \cup \{y_R^- : R \in \cR\})$, as follows:
\[
p(x)=\begin{cases}
x,&\text{if }x\in {}^\bullet \mathbb{R}-(\{x_L^+ : L \in \cL\} \cup \{y_R^- : R \in \cR\})\\
x_L,&\text{if }x=x_L^+\\
y_R,&\text{if }x=y_R^-
\end{cases}
\]

Now, define an order $<_F$ on ${}^\bullet \mathbb{R}$, so that:
\[
x<_Fy \Leftrightarrow \begin{cases}
p(x) \triangleL  p(y) \\
or\\
x=x_L \textrm{ and } y = x_L^+\\
or\\
x=y_R^- \textrm{ and } y = y_R
\end{cases}
\]

Obviously, $<_F$ is a linear order and the restriction
of $<_F$ to ${}^\bullet \mathbb{R}-(\{x_L^+ : L \in \cL\} \cup \{y_R^- : R \in \cR\})$ equals $\triangleLq$, the order in Definition \ref{definition-order Giordano}. In addition, we can set $x_L^+ = x_L + h$, where $h$ is not zero in $\FR$ and $h>0$, that is,
 $\lim_{t \to 0^+} h_t/t \neq 0$ and, respectively, we set $x_R^- = x_R-h$, and this completes the proof.
\end{proof}

\begin{theorem}
$\FR$ equipped with the order topology from $<_F$ is a LOTS.
\end{theorem}
\begin{proof}
We will now show that the topology $\cT$ on ${}^\bullet \mathbb{R}-(\{x_L^+ : L \in \cL\} \cup \{y_R^- : R \in \cR\})$ coincides with the
subspace topology on ${}^\bullet \mathbb{R}-(\{x_L^+ : L \in \cL\} \cup \{y_R^- : R \in \cR\})$ that is inherited from the $<_F$-order topology
on ${}^\bullet \mathbb{R}$.

But, since $\cL \cup \cR$ forms a subbasis for $\cT$, that consists of two nests, every
set in $\cT$ can be written as a union of sets of the form $L \cap R$, where
$L \in \cL$ and $R \in \cR$. It suffices therefore to show that every $L \in \cL$ and
$R \in \cR$ can be written as the intersection of an order-open set with ${}^\bullet \mathbb{R}-(\{x_L^+ : L \in \cL\} \cup \{y_R^- : R \in \cR\})$.
But this is always true, since if $L \in \cL$, with $\triangleLq$-maximal element $x_L$, then
$L = {}^\bullet \mathbb{R}-(\{x_L^+ : L \in \cL\} \cup \{y_R^- : R \in \cR\}) \cap \{ x \in \FR: x<_F x_L^+\}$.

The argument for $R \in \cR$ is similar, and this completes the proof.
\end{proof}

\subsection{Remarks.}
\begin{enumerate}
\item The order topology $\cT_{<_F}$ equals the topology $\cT_{\cL_{<_F} \cup \cR_{<_F}}$,
where $\cL_{<_F}=\{k \in \FR : k <_F l\}$, some $l \in \FR$ and $\cR_{<_F}=\{k \in \FR : l <_F k\}$, some $l \in \FR$.
This is because $\cL_{<_F} \cup \cR_{<_F}$ $T_1$-separates $\FR$ and both $\cL_{<_F}$ and $\cR_{<_F}$ are interlocking nests. So, unlike
the GO-space topology $\cT_\le$ on $\FR$, where $\cT_\le \subset \cT_{\cL \cup \cR}$, $<_F$ provides a natural extension
of the natural linear order of the set of real numbers to the Fermat real line and the order topology from $<_F$ can
be completely described via the nests $\cL_{<_F}$ and $\cR_{<_F}$.

\item Viewing the Fermat real line as a LOTS and working with nests $\cL_{<_F}$ and $\cR_{<_F}$, one can now
define the product topology for $\FR^n$, some positive integer $n$, or even more generaly for $\Pi_{i \in I} \FR_i$, some arbitrary indexing
set $I$, in the usual way via the subbasis $\pi_{j_0}^{-1}(A_{j_0}) = \Pi_{i \in I}\{\FR_i : i \neq j_0\} \times A_{j_0}$,
where $A_{j_0}$ is an open subset in the coordinate space $\FR_{j_0}$ in the order topology $\cT_{<_F}$ and
$\pi_i : \Pi_{i \in I}\FR_i \to \FR_i$ the projection.

\item In this way one can define continuity for
any function $f$ from a topological space $Y$ into the product space $\Pi_{i \in I} \FR_i$ via
the continuity of $\pi_i \circ f :Y \to \FR_i$.

\item The neight of $\FR$ is $2$ and the neight of $\FR^n = n+1$ (see \cite{ref10}).
Using the product topology, as stated in Remark (2), we use four nests in order to
define -for example- the topology in $\FR^2$, but since the neight of $\FR^2$ is $3$,
one can define a topology using three nests exclusively.

\end{enumerate}

\subsection{Questions.}
\begin{enumerate}

\item As a LOTS, $(\FR,<_F)$ has rich topological properties. It is, for example,
a monotone normal space. It would be interesting though to have an extensive study on
the metrizability of this space. It is known that in a GO-space the terms
metrizable, developable, semistratifiable, etc. are equivalent (see \cite{ref8} and \cite{ref9}).
The real line (i.e. the set of all standard reals, from the point of view of $\FR$)
is a developable LOTS and this is equivalent to say that it is also a metrizable LOTS.
Is $(\FR,\cT_{<_F})$ developable?

\item Which of the subspaces of $(\FR,\cT_{<_F})$ are developable?
\end{enumerate}

Since any sequence $x_1,x_2,\cdots$ of points in $\Pi_{i \in I} \FR_i$ will converge to a point
$x \in \Pi_{i \in I}\FR_i$, iff for every projection $\pi_i: \Pi_{i \in I} \FR_i \to \FR_i$ the
sequence $\pi_i(x_1), \pi_i(x_2),\cdots$ converges to $\pi_i(x)$ in the coordinate space $\FR_i$,
any answer to the above questions will be foundamental towards our understanding of convergence in the ring
of Fermat Reals.

\newpage


\end{document}